\documentclass{amsart}
\usepackage{amssymb}
\usepackage{amsfonts}
\usepackage{geometry}

\newtheorem{theorem}{Theorem}
\theoremstyle{plain}

\newtheorem{remark}[theorem]{Remark}

\numberwithin{equation}{section}
\input{tcilatex}
\begin{document}
\title[restricted weak type]{A restricted weak type inequality with
application to a $Tp$ theorem and optimal cancellation conditions for $CZO$'s%
}
\author[E.T. Sawyer]{Eric T. Sawyer}
\address{ Department of Mathematics \& Statistics, McMaster University, 1280
Main Street West, Hamilton, Ontario, Canada L8S 4K1 }

\begin{abstract}
This paper continues the investigation begun in \texttt{arXiv:1906.05602} of
extending the $T1\,$theorem of David and Journ\'{e}, and optimal
cancellation conditions, to more general weight pairs. The main additional
tool developed here is a two weight restricted weak type inequality, which
eliminates the $BICT$ from \texttt{arXiv:1906.05602} under certain
retrictions.

Assume $\sigma $ and $\omega $ are locally finite positive Borel measures on 
$\mathbb{R}^{n}$, with one of them an $A_{\infty }$ weight, and suppose that 
$T^{\alpha }$ is an $\alpha $-fractional Calder\'{o}n-Zygmund singular
integral in $\mathbb{R}^{n}$, $0\leq \alpha <n$. In the case $\alpha =0$ we
also assume that $T^{0}$ is bounded on unweighted $L^{2}\left( \mathbb{R}%
^{n}\right) $. Then the two weight restricted weak type inequality for $%
T^{\alpha }$,%
\begin{eqnarray*}
&&\left\vert \int_{F}T^{\alpha }\left( \mathbf{1}_{E}\sigma \right) d\omega
\right\vert \leq C\sqrt{\left\vert E\right\vert _{\sigma }\left\vert
F\right\vert _{\omega }}, \\
&&\text{for all compact subsets }E,F\text{ of a cube }Q,
\end{eqnarray*}%
holds \emph{if} \emph{and }(provided $T^{\alpha }$ is elliptic) \emph{only if%
} the classical fractional Muckenhoupt condition $A_{2}^{\alpha }$ holds,
i.e.%
\begin{equation*}
\left\vert Q\right\vert _{\sigma }\left\vert Q\right\vert _{\omega }\leq
C^{\prime }A_{2}^{\alpha }\left\vert Q\right\vert ^{2-\frac{2\alpha }{n}},\
\ \ \ \text{\ for all cubes }Q.
\end{equation*}

Applications are then given to a $Tp$ theorem for $\alpha $-fractional $CZO$%
's $T^{\alpha }$ and doubling measures when one of the weights is $A_{\infty
}$ (and $T^{0}$ is bounded on unweighted $L^{2}\left( \mathbb{R}^{n}\right) $
when $\alpha =0$), and then to optimal cancellation conditions for such $CZO$%
's in terms of polynomial testing in similar situations.
\end{abstract}

\dedicatory{In memory of Professor Elias M. Stein.}
\email{sawyer@mcmaster.ca}
\maketitle
\tableofcontents

\section{The restricted weak type theorem}

Let $\sigma $ and $\omega $ be locally finite positive Borel measures on $%
\mathbb{R}^{n}$. For $0\leq \alpha <n$, the classical $\alpha $-fractional
Muckenhoupt condition for the measure pair $\left( \sigma ,\omega \right) $
is given by%
\begin{equation}
A_{2}^{\alpha }\left( \sigma ,\omega \right) \equiv \sup_{Q\in \mathcal{P}%
^{n}}\frac{\left\vert Q\right\vert _{\sigma }}{\left\vert Q\right\vert ^{1-%
\frac{\alpha }{n}}}\frac{\left\vert Q\right\vert _{\omega }}{\left\vert
Q\right\vert ^{1-\frac{\alpha }{n}}}<\infty ,  \label{frac Muck}
\end{equation}%
and the stronger one-tailed $\alpha $-fractional Muckenhoupt condition is
given by%
\begin{equation}
\mathcal{A}_{2}^{\alpha }\left( \sigma ,\omega \right) \equiv \sup_{Q\in 
\mathcal{Q}^{n}}\mathcal{P}^{\alpha }\left( Q,\sigma \right) \frac{%
\left\vert Q\right\vert _{\omega }}{\left\vert Q\right\vert ^{1-\frac{\alpha 
}{n}}}<\infty \text{ and }\mathcal{A}_{2}^{\alpha ,\ast }\left( \sigma
,\omega \right) \equiv \mathcal{A}_{2}^{\alpha }\left( \omega ,\sigma
\right) <\infty ,  \label{one-tailed}
\end{equation}%
where $\mathcal{P}^{\alpha }\left( Q,\sigma \right) $ is the reproducing
Poisson integral

\begin{equation*}
\mathcal{P}^{\alpha }\left( Q,\mu \right) \equiv \int_{\mathbb{R}^{n}}\left( 
\frac{\left\vert Q\right\vert ^{\frac{1}{n}}}{\left( \left\vert Q\right\vert
^{\frac{1}{n}}+\left\vert x-x_{Q}\right\vert \right) ^{2}}\right) ^{n-\alpha
}d\mu \left( x\right) .
\end{equation*}

The measure $\sigma $ is said to be doubling if there is a pair of constants 
$\left( \beta ,\gamma \right) \in \left( 0,1\right) ^{2}$, called doubling
parameters, such that%
\begin{equation}
\left\vert \beta Q\right\vert _{\mu }\geq \gamma \left\vert Q\right\vert
_{\mu }\ ,\ \ \ \ \ \text{for all cubes }Q\in \mathcal{P}^{n}.
\label{def rev doub}
\end{equation}%
A familiar equivalent reformulation of (\ref{def rev doub}) is that there is
a positive constant $C_{\limfunc{doub}}$, called the doubling constant, such
that $\left\vert 2Q\right\vert _{\mu }\leq C_{\limfunc{doub}}\left\vert
Q\right\vert _{\mu }$ for all cubes $Q\in \mathcal{P}^{n}$. The absolutely
continuous measure $d\omega \left( x\right) =w\left( x\right) dx$ is said to
be an $A_{\infty }$ weight if there are constants $0<\varepsilon ,\eta <1$,
called $A_{\infty }$ parameters, such that%
\begin{equation*}
\frac{\left\vert E\right\vert _{\omega }}{\left\vert Q\right\vert _{\omega }}%
<\eta \text{ whenever }E\text{ compact }\subset Q\text{ a cube with }\frac{%
\left\vert E\right\vert }{\left\vert Q\right\vert }<\varepsilon .
\end{equation*}%
A useful reformulation given in \cite[Theorem III on page 244]{CoFe} is that
there is $C>0$ and an $A_{\infty }$ exponent $\varepsilon >0$ such that%
\begin{equation}
\frac{\left\vert E\right\vert _{\omega }}{\left\vert Q\right\vert _{\omega }}%
\leq C\left( \frac{\left\vert E\right\vert }{\left\vert Q\right\vert }%
\right) ^{\varepsilon }\text{ whenever }E\text{ compact }\subset Q\text{ a
cube}.  \label{reform}
\end{equation}

Let $0\leq \alpha <n$. We define a standard $\alpha $-fractional CZ kernel $%
K^{\alpha }(x,y)$ to be a function $K^{\alpha }:\mathbb{R}^{n}\times \mathbb{%
R}^{n}\rightarrow \mathbb{R}$ satisfying the following fractional size and\
minimal smoothness conditions: There is $\delta >0$ and $C_{CZ}>0$ such that
for $x\neq y$, 
\begin{eqnarray}
\left\vert K^{\alpha }\left( x,y\right) \right\vert &\leq &C_{CZ}\left\vert
x-y\right\vert ^{\alpha -n}  \label{sizeandsmoothness'} \\
\left\vert K^{\alpha }\left( x,y\right) -K^{\alpha }\left( x^{\prime
},y\right) \right\vert &\leq &C_{CZ}\left( \frac{\left\vert x-x^{\prime
}\right\vert }{\left\vert x-y\right\vert }\right) ^{\delta }\left\vert
x-y\right\vert ^{\alpha -n},\ \ \ \ \ \frac{\left\vert x-x^{\prime
}\right\vert }{\left\vert x-y\right\vert }\leq \frac{1}{2},  \notag
\end{eqnarray}%
and where the same inequalities hold for the adjoint kernel $K^{\alpha ,\ast
}\left( x,y\right) \equiv K^{\alpha }\left( y,x\right) $, in which $x$ and $%
y $ are interchanged. We also consider vector kernels $K^{\alpha }=\left(
K_{j}^{\alpha }\right) $ where each $K_{j}^{\alpha }$ is as above, often
without explicit mention. This includes for example the vector Riesz
transform in higher dimensions. Given a standard $\alpha $-fractional $CZ$
kernel $K^{\alpha }$, we consider truncated kernels $K_{\delta ,R}^{\alpha
}\left( x,y\right) =\eta _{\delta ,R}^{\alpha }\left( \left\vert
x-y\right\vert \right) K^{\alpha }\left( x,y\right) $ which uniformly
satisfy (\ref{sizeandsmoothness'}). Then the truncated operator $T_{\delta
,R}^{\alpha }$ with kernel $K_{\delta ,R}^{\alpha }$ is pointwise
well-defined, and we will refer to the pair $T^{\alpha }=\left( K^{\alpha
},\left\{ \eta _{\delta ,R}^{\alpha }\right\} _{0<\delta <R<\infty }\right) $
as an $\alpha $-fractional singular integral operator.

Let $T_{\sigma }^{\alpha }f=T^{\alpha }\left( f\sigma \right) $. We say that
an $\alpha $-fractional singular integral operator $T^{\alpha }$ satisfies
the restricted weak type inequality relative to the measure pair $\left(
\sigma ,\omega \right) $ provided%
\begin{eqnarray}
&&\mathfrak{N}_{T^{\alpha }}^{\limfunc{restricted}\limfunc{weak}}\left(
\sigma ,\omega \right) \equiv \sup_{Q\in \mathcal{P}^{n}}\sup_{E,F\subset Q}%
\frac{1}{\sqrt{\left\vert E\right\vert _{\sigma }\left\vert F\right\vert
_{\omega }}}\left\vert \int_{F}T_{\sigma }^{\alpha }\left( \mathbf{1}%
_{E}\right) \omega \right\vert <\infty ,  \label{RWT} \\
&&\text{where the second sup is taken over all compact subsets }E,F\ \text{%
of the cube }Q,  \notag \\
&&\text{and where }0<\delta <R<\infty .  \notag
\end{eqnarray}%
In the presence of the classical Muckenhoupt condition $A_{2}^{\alpha }$,
the restricted weak type inequality in (\ref{RWT}) is essentially
independent of the choice of truncations used - see \cite{LaSaShUr3}.
Finally, as in \cite{SaShUr7}, an $\alpha $-fractional vector Calder\'{o}%
n-Zygmund kernel $K^{\alpha }=\left( K_{j}^{\alpha }\right) $ is said to be
elliptic if there is $c>0$ such that for each unit vector $\mathbf{u}\in 
\mathbb{R}^{n}$ there is $j$ satisfying%
\begin{equation*}
\left\vert K_{j}^{\alpha }\left( x,x+tu\right) \right\vert \geq ct^{\alpha
-n},\ \ \ \ \ \text{for all }t>0.
\end{equation*}

\begin{remark}
In the special case $\alpha =0$, we will make the additional assumption
throughout this paper that $T^{0}$ is bounded on \emph{unweighted} $%
L^{2}\left( \mathbb{R}^{n}\right) $. This is done in order to be able to use
the weak type $\left( 1,1\right) $ result on Lebesgue measure for maximal
truncations of such operators, that follows from standard Calder\'{o}%
n-Zygmund theory as in \cite[Corollary 2 on page 36]{Ste2}. As a
consequence, our results say nothing new in the case of \emph{equal} weights
when $\alpha =0$. Of course, by the $T1$ theorem in \cite{DaJo}, the
additional assumption of boundedness on unweighted $L^{2}\left( \mathbb{R}%
^{n}\right) $ is equivalent to the classical cube testing conditions on
Lebesgue measure.
\end{remark}

\begin{theorem}
\label{restricted weak type}Suppose that $\sigma $ and $\omega $ are locally
finite positive Borel measures on $\mathbb{R}^{n}$, with one of them an $%
A_{\infty }$ weight. Let $0\leq \alpha <n$. Suppose also that $T^{\alpha }$
is a standard $\alpha $-fractional Calder\'{o}n-Zygmund singular integral in 
$\mathbb{R}^{n}$, and that when $\alpha =0$ the operator $T^{0}$ is bounded
on unweighted $L^{2}\left( \mathbb{R}^{n}\right) $. Then the two weight
restricted weak type inequality for $T^{\alpha }$ relative to the measure
pair $\left( \sigma ,\omega \right) $ holds \emph{if }the classical
fractional Muckenhoupt constant $A_{2}^{\alpha }$ in (\ref{frac Muck}) is
finite and, provided $T^{\alpha }$ is elliptic, \emph{only if} $%
A_{2}^{\alpha }$ is finite. Moreover, in the case $T^{\alpha }$ is elliptic, 
\begin{equation*}
\mathfrak{N}_{T^{\alpha }}^{\limfunc{restricted}\limfunc{weak}}\left( \sigma
,\omega \right) \approx A_{2}^{\alpha }\left( \sigma ,\omega \right) ,
\end{equation*}%
where the implied constants depend on the Calderon-Zygmund norm $C_{CZ}$ in (%
\ref{sizeandsmoothness'}) and the $A_{\infty }$ norm of one of the weights.
\end{theorem}

\begin{remark}
The proof of the theorem shows a bit more, namely that the restricted weak
type norms of $T^{\alpha }$ and its maximal trunction operator $T_{\flat
}^{\alpha }$ (see below) are equivalent, and including the fractional
integral $I^{\alpha }$ (see below) when $0<\alpha <n$.
\end{remark}

\subsection{Proof of the restricted weak type theorem}

The proof of the theorem is a standard application of an idea from four and
a half decades ago, namely the 1973 $\limfunc{good}-\lambda $ inequality of
Burkholder\ \cite{Bur}, and specifically the 1974 inequality of R. Coifman
and C. Fefferman \cite{CoFe}. This\ latter inequality relates maximal
truncations of a CZ singular integral to the maximal operator $M$, which we
now briefly recall.

Given an $\alpha $-fractional CZ operator $T^{\alpha }$, define the maximal
truncation operator $T_{\flat }^{\alpha }$ by%
\begin{equation*}
T_{\flat }^{\alpha }\left( f\sigma \right) \left( x\right) \equiv
\sup_{0<\varepsilon <R<\infty }\left\vert \int_{\left\{ \varepsilon
<\left\vert y\right\vert <R\right\} }K^{\alpha }\left( x,y\right) f\left(
y\right) d\sigma \left( y\right) \right\vert ,\ \ \ \ \ x\in \mathbb{R}^{n},
\end{equation*}%
for any locally finite positive Borel measure $\sigma $ on $\mathbb{R}^{n}$,
and $f\in L^{2}\left( \sigma \right) $. Define the $\alpha $-fractional
Hardy-Littlewood maximal operator $M^{\alpha }$ by%
\begin{equation*}
M^{\alpha }\left( f\sigma \right) \left( x\right) \equiv \sup_{Q\in \mathcal{%
P}^{n}:\ x\in Q}\frac{1}{\left\vert Q\right\vert ^{1-\frac{\alpha }{n}}}%
\int_{Q}\left\vert f\right\vert d\sigma ,\ \ \ \ \ x\in \mathbb{R}^{n},
\end{equation*}%
where here we may take the cubes $Q$ in the supremum to be closed.

Let $\omega $ be an $A_{\infty }$ weight. Suppose first that $\alpha =0$.
Then the Coifman-Fefferman $\limfunc{good}-\lambda $ inequality in \cite[see
inequality (7) on page 245]{CoFe} is%
\begin{equation}
\left\vert \left\{ x\in Q:\ T_{\flat }\left( f\sigma \right) \left( x\right)
>2\lambda \text{ and }M\left( f\sigma \right) \left( x\right) \leq \beta
\lambda \right\} \right\vert _{\omega }\leq C\beta ^{\varepsilon }\left\vert
\left\{ x\in Q:\ T_{\flat }\left( f\sigma \right) \left( x\right) >\lambda
\right\} \right\vert _{\omega },  \label{CF good}
\end{equation}%
for all $\lambda >0$, where $\varepsilon >0$ is the $A_{\infty }$ exponent
in (\ref{reform}). The kernels considered in \cite{CoFe} are convolution
kernels with order $1$ smoothness and bounded Fourier transform. However,
since we are assuming here that $T$ is bounded on unweighted $L^{2}\left( 
\mathbb{R}^{n}\right) $, standard CZ theory \cite[Corollary 2 on page 36]%
{Ste2} implies that $T_{\flat }$ is weak type $\left( 1,1\right) $ on
Lebesgue measure. This estimate is the key to the proof in \cite[see pages
245-246 where the the weak type $\left( 1,1\right) $ inequality for $%
T_{\flat }$ is used]{CoFe}, and this proof shows that the kernel of the
operator $T$ may be taken to be a standard kernel in the above sense.

In the case $0<\alpha <n$, this $\limfunc{good}-\lambda $ inequality for an $%
A_{\infty }$ weight $\omega $ was extended in \cite{MuWh} (by essentially
the same proof) when $T_{\flat }$ and $M$ are replaced by $I^{\alpha }$ and $%
M^{\alpha }$ respectively:%
\begin{equation}
\left\vert \left\{ x\in Q:\ I^{\alpha }\left( f\sigma \right) \left(
x\right) >2\lambda \text{ and }M^{\alpha }\left( f\sigma \right) \left(
x\right) \leq \beta \lambda \right\} \right\vert _{\omega }\leq \frac{C}{%
\beta }\left\vert \left\{ x\in Q:\ I^{\alpha }\left( f\sigma \right) \left(
x\right) >\lambda \right\} \right\vert _{\omega },  \label{CF good alpha}
\end{equation}%
for all $\lambda >0$. Here the fractional integral $I^{\alpha }$ is given by 
$I^{\alpha }\nu \left( x\right) \equiv \int_{\mathbb{R}^{n}}\left\vert
x-y\right\vert ^{\alpha -n}d\nu \left( y\right) $, and we will use below the
obvious fact that $\left\vert T_{\flat }^{\alpha }\nu \left( x\right)
\right\vert \leq CI^{\alpha }\nu \left( x\right) $ for $d\nu \geq 0$. ($%
I_{\alpha }$ is a positive operator satisfying the weak type $\left( 1,\frac{%
n}{n-\alpha }\right) $\ inequality on Lebesgue measure, and this is why
there is no need to assume any additional unweighted boundedness of $%
T^{\alpha }$ when $\alpha >0$).

From such $\limfunc{good}-\lambda $ inequalities for $A_{\infty }$ weights $%
\omega $, standard arguments show that $\left\Vert T_{\flat }^{\alpha
}\left( f\sigma \right) \right\Vert _{L^{2}\left( \omega \right) }\lesssim $ 
$\left\Vert M^{\alpha }\left( f\sigma \right) \right\Vert _{L^{2}\left(
\omega \right) }$ for $0\leq \alpha <n$ and $f\in L^{2}\left( \sigma \right) 
$. We will use a weak type variant of this latter inequality, together with
the equivalence of $\mathfrak{N}_{M^{\alpha }}^{\limfunc{weak}}\left( \sigma
,\omega \right) $ and $A_{2}^{\alpha }\left( \sigma ,\omega \right) $, to
prove the theorem.

\begin{proof}[Proof of Theorem \protect\ref{restricted weak type}]
Since the restricted weak type inequality is self-dual, we can assume
without loss of generality that $\omega $ is an $A_{\infty }$ weight. We
begin by showing that the $\limfunc{good}-\lambda $ inequalities for $%
A_{\infty }$ weights $\omega $ imply \emph{weak type} control, exercising
care in absorbing terms. Indeed, for $t>0$, we obtain from (\ref{CF good})
and (\ref{CF good alpha}) that%
\begin{eqnarray*}
\sup_{0<\lambda \leq t}\lambda ^{2}\left\vert \left\{ T_{\flat }^{\alpha
}\left( f\sigma \right) >\lambda \right\} \right\vert _{\omega }
&=&4\sup_{0<\lambda \leq \frac{t}{2}}\lambda ^{2}\left\vert \left\{ T_{\flat
}^{\alpha }\left( f\sigma \right) >2\lambda \right\} \right\vert _{\omega }
\\
&\leq &4\sup_{0<\lambda \leq \frac{t}{2}}\lambda ^{2}\left\{ \left\vert
\left\{ M^{\alpha }\left( f\sigma \right) >\beta \lambda \right\}
\right\vert _{\omega }+\frac{C}{\beta }\left\vert \left\{ T_{\flat }^{\alpha
}\left( f\sigma \right) >\lambda \right\} \right\vert _{\omega }\right\} \\
&=&\frac{4}{\beta ^{2}}\sup_{0<\lambda \leq \frac{\beta }{2}t}\lambda
^{2}\left\vert \left\{ M^{\alpha }\left( f\sigma \right) >\beta \lambda
\right\} \right\vert _{\omega }+4\sup_{0<\lambda \leq \frac{t}{2}}\frac{C}{%
\beta }\left\vert \left\{ T_{\flat }^{\alpha }\left( f\sigma \right)
>\lambda \right\} \right\vert _{\omega } \\
&\leq &\frac{4}{\beta ^{2}}\left\Vert M^{\alpha }\left( f\sigma \right)
\right\Vert _{L^{2,\infty }\left( \omega \right) }^{2}+\frac{4C}{\beta }%
\sup_{0<\lambda \leq t}\lambda ^{2}\left\vert \left\{ T_{\flat }^{\alpha
}\left( f\sigma \right) >\lambda \right\} \right\vert _{\omega }\ .
\end{eqnarray*}%
Now choose $\beta $ so that $\frac{4C}{\beta }=\frac{1}{2}$. Provided that $%
\sup_{0<\lambda \leq t}\lambda ^{2}\left\vert \left\{ T_{\flat }^{\alpha
}\left( f\sigma \right) >\lambda \right\} \right\vert _{\omega }$ is finite
for each $t>0$, we can absorb the final term on the right into the left hand
side to obtain%
\begin{equation*}
\sup_{0<\lambda \leq t}\lambda ^{2}\left\vert \left\{ T_{\flat }^{\alpha
}\left( f\sigma \right) >\lambda \right\} \right\vert _{\omega }\leq \frac{8%
}{\beta ^{2}}\left\Vert M^{\alpha }\nu \right\Vert _{L^{2,\infty }\left(
\omega \right) }^{2},\ \ \ \ \ t>0,
\end{equation*}%
which gives%
\begin{equation*}
\left\Vert T_{\flat }^{\alpha }\left( f\sigma \right) \right\Vert
_{L^{2,\infty }\left( \omega \right) }^{2}=\sup_{0<\lambda <\infty }\lambda
^{2}\left\vert \left\{ T_{\flat }^{\alpha }\left( f\sigma \right) >\lambda
\right\} \right\vert _{\omega }\leq \frac{8}{\beta ^{2}}\left\Vert M^{\alpha
}\left( f\sigma \right) \right\Vert _{L^{2,\infty }\left( \omega \right)
}^{2}\ .
\end{equation*}

Suppose now that $\alpha =0$. In order to obtain finiteness of the supremum
over $0<\lambda \leq t$, we take $f\in L^{2}\left( \sigma \right) $ wtih $%
\left\vert f\right\vert \leq 1$ and $\limfunc{supp}f\subset B\left(
0,r\right) $ with $1\leq r<\infty $ and $\left\vert B\left( 0,r\right)
\right\vert _{\sigma }>0$. Then if $x\notin B\left( 0,2r\right) $, we have $%
\left\vert K\left( x,y\right) \right\vert \leq C_{CZ}r^{-n}$ and hence%
\begin{equation*}
T_{\flat }\left( f\sigma \right) \left( x\right) =\sup_{0<\varepsilon
<R<\infty }\left\vert \int_{\left\{ \varepsilon <\left\vert y\right\vert
<R\right\} \cap B\left( 0,r\right) }K\left( x,y\right) f\left( y\right)
d\sigma \left( y\right) \right\vert \leq C_{CZ}\left( \frac{2}{\left\vert
x\right\vert }\right) ^{n}\left\vert B\left( 0,r\right) \right\vert _{\sigma
}\ .
\end{equation*}%
This shows that 
\begin{eqnarray*}
&&\sup_{0<\lambda \leq t}\lambda ^{2}\left\vert \left\{ T_{\flat }\nu
>\lambda \right\} \right\vert _{\omega }\leq t^{2}\left\vert B\left(
0,2r\right) \right\vert _{\omega }+\sup_{0<\lambda <C_{CZ}r^{-n}\left\vert
B\left( 0,r\right) \right\vert _{\sigma }}\lambda ^{2}\left\vert \left\{
T_{\flat }\nu >\lambda \right\} \setminus B\left( 0,2r\right) \right\vert
_{\omega } \\
&\leq &t^{2}\left\vert B\left( 0,2r\right) \right\vert _{\omega
}+\sup_{0<\lambda \leq t}\lambda ^{2}\left\vert \left\{ C_{CZ}\left( \frac{2%
}{\left\vert x\right\vert }\right) ^{n}\left\vert B\left( 0,r\right)
\right\vert >\lambda \right\} \right\vert _{\omega } \\
&=&t^{2}\left\vert B\left( 0,2r\right) \right\vert _{\omega
}+\sup_{0<\lambda \leq t}\lambda ^{2}\left\vert B\left( 0,\gamma _{\lambda
}r\right) \right\vert _{\omega }\ ,
\end{eqnarray*}%
with $\gamma _{\lambda }\equiv 2\sqrt[n]{\frac{C_{CZ}c}{\lambda }}$, since $%
\left\{ C_{CZ}\left( \frac{2}{\left\vert x\right\vert }\right)
^{n}\left\vert B\left( 0,r\right) \right\vert >\lambda \right\} =B\left( 0,2%
\sqrt[n]{\frac{C_{CZ}c}{\lambda }}r\right) $.

On the other hand, the $A_{2}$ condition implies that for $\lambda \leq
\lambda _{0}\equiv C_{CZ}c$, we have $\gamma _{\lambda }\geq \gamma
_{\lambda _{0}}=2$ so that 
\begin{equation*}
\left\vert B\left( 0,\gamma _{\lambda }r\right) \right\vert _{\omega
}\lesssim A_{2}\left( \sigma ,\omega \right) \frac{\left\vert B\left(
0,\gamma _{\lambda }r\right) \right\vert ^{2}}{\left\vert B\left( 0,\gamma
_{\lambda }r\right) \right\vert _{\sigma }}\leq A_{2}\left( \sigma ,\omega
\right) \frac{\left( \gamma _{\lambda }r\right) ^{2n}}{\left\vert B\left(
0,2r\right) \right\vert _{\sigma }}A_{2}\left( \sigma ,\omega \right) =\frac{%
4^{n}\left( \frac{C_{CZ}c}{\lambda }\right) ^{2}}{\left\vert B\left(
0,2r\right) \right\vert _{\sigma }},
\end{equation*}%
and hence%
\begin{equation*}
\lambda ^{2}\left\vert B\left( 0,\gamma _{\lambda }r\right) \right\vert
_{\omega }\leq \lambda ^{2}\frac{4^{n}\left( \frac{C_{CZ}c}{\lambda }\right)
^{2}}{\left\vert B\left( 0,2r\right) \right\vert _{\sigma }}=\frac{%
4^{n}C_{CZ}c^{2}}{\left\vert B\left( 0,2r\right) \right\vert _{\sigma }},\ \
\ \ \ \text{for }\lambda \leq \lambda _{0}\ .
\end{equation*}

Finally we have%
\begin{equation*}
\sup_{\lambda _{0}<\lambda \leq t}\lambda ^{2}\left\vert B\left( 0,\gamma
_{\lambda }r\right) \right\vert _{\omega }\leq t^{2}\left\vert B\left(
0,\gamma _{\lambda _{0}}r\right) \right\vert _{\omega }=t^{2}\left\vert
B\left( 0,2r\right) \right\vert _{\omega }\ ,
\end{equation*}%
and altogether then%
\begin{equation*}
\sup_{0<\lambda \leq t}\lambda ^{2}\left\vert \left\{ T_{\flat }\nu >\lambda
\right\} \right\vert _{\omega }\leq t^{2}\left\vert B\left( 0,2r\right)
\right\vert _{\omega }+\frac{4^{n}C_{CZ}c^{2}}{\left\vert B\left(
0,2r\right) \right\vert _{\sigma }}+t^{2}\left\vert B\left( 0,2r\right)
\right\vert _{\omega }
\end{equation*}%
which is finite for $0<t<\infty $. Thus we conclude that 
\begin{equation*}
\mathfrak{N}_{T}^{\limfunc{restricted}\limfunc{weak}}\left( \sigma ,\omega
\right) \leq \mathfrak{N}_{T}^{\limfunc{weak}}\left( \sigma ,\omega \right)
\leq \mathfrak{N}_{T_{\flat }}^{\limfunc{weak}}\left( \sigma ,\omega \right)
\lesssim \mathfrak{N}_{M}^{\limfunc{weak}}\left( \sigma ,\omega \right)
\approx A_{2}\left( \sigma ,\omega \right) ,
\end{equation*}%
where the final equivalence is well known, and can be obtained by averaging
over dyadic grids $\mathcal{D}$ the inequality $\mathfrak{N}_{M_{\mathcal{D}%
}^{\alpha }}^{\limfunc{weak}}\left( \sigma ,\omega \right) \lesssim
A_{2}\left( \sigma ,\omega \right) $ for dyadic operators 
\begin{equation*}
M_{\mathcal{D}}^{\alpha }f\left( x\right) \equiv \sup_{Q\in \mathcal{P}%
^{n}:\ x\in Q}\frac{1}{\left\vert Q\right\vert ^{1-\frac{\alpha }{n}}}%
\int_{Q}\left\vert f\right\vert d\sigma .
\end{equation*}%
The dyadic inequality is in turn an immediate consequence of the dyadic
covering lemma. Conversely, if $T^{\alpha }$ is elliptic, then $A_{2}\left(
\sigma ,\omega \right) \lesssim \mathfrak{N}_{T}^{\limfunc{restricted}%
\limfunc{weak}}\left( \sigma ,\omega \right) $ (see \cite{LiTr} and \cite%
{SaShUr7}).

The same sort of arguments give the analogous inequality when $0<\alpha <n$: 
\begin{equation*}
\mathfrak{N}_{T^{\alpha }}^{\limfunc{restricted}\limfunc{weak}}\left( \sigma
,\omega \right) \leq \mathfrak{N}_{I^{\alpha }}^{\limfunc{restricted}%
\limfunc{weak}}\left( \sigma ,\omega \right) \leq \mathfrak{N}_{I^{\alpha
}}^{\limfunc{weak}}\left( \sigma ,\omega \right) \lesssim \mathfrak{N}%
_{M^{\alpha }}^{\limfunc{weak}}\left( \sigma ,\omega \right) \approx
A_{2}^{\alpha }\left( \sigma ,\omega \right) ,
\end{equation*}%
and conversely, $A_{2}^{\alpha }\left( \sigma ,\omega \right) \lesssim 
\mathfrak{N}_{T^{\alpha }}^{\limfunc{restricted}\limfunc{weak}}\left( \sigma
,\omega \right) $ if $T^{\alpha }$ is elliptic. This completes the proof of
Theorem \ref{restricted weak type}.
\end{proof}

\section{A $Tp$ theorem for doubling weights when one weight is $A_{\infty }$%
}

The \emph{Bilinear Indicator/Cube Testing} property introduced in \cite{Saw2}
is%
\begin{equation}
\mathcal{BICT}_{T^{\alpha }}\left( \sigma ,\omega \right) \equiv \sup_{Q\in 
\mathcal{P}^{n}}\sup_{E,F\subset Q}\frac{1}{\sqrt{\left\vert Q\right\vert
_{\sigma }\left\vert Q\right\vert _{\omega }}}\left\vert \int_{F}T_{\sigma
}^{\alpha }\left( \mathbf{1}_{E}\right) \omega \right\vert <\infty ,
\label{def ind WBP}
\end{equation}%
where the second supremum above is taken over all compact sets $E$ and $F$
contained in the cube $Q$. Theorem \ref{restricted weak type} shows that%
\begin{equation*}
\mathcal{BICT}_{T^{\alpha }}\left( \sigma ,\omega \right) \lesssim
A_{2}^{\alpha }\left( \sigma ,\omega \right) ,
\end{equation*}%
with the implied constant depending on $T^{\alpha }$ and the $A_{\infty }$
parameters of one of the weights. This latter inequality shows that we can
immediately remove $\mathcal{BICT}_{T^{\alpha }}\left( \sigma ,\omega
\right) $ from the right hand side of Theorem 1 in \cite{Saw2}, when in
addition one of the weights is $A_{\infty }$ (and $T^{0}$ is bounded on
unweighted $L^{2}\left( \mathbb{R}^{n}\right) $ if $\alpha =0$). We now
describe the resulting theorem.

First, for $0\leq \alpha <n$ and $\kappa _{1},\kappa _{2}\in \mathbb{N}$, we
say that $K^{\alpha }\left( x,y\right) $ is a standard $\left( \kappa
_{1}+\delta ,\kappa _{2}+\delta \right) $-smooth $\alpha $-fractional kernel
if for $x\neq y$, and with $\nabla _{1}$ denoting gradient in the first
variable, and $\nabla _{2}$ denoting gradient in the second variable, 
\begin{eqnarray*}
&&\left\vert \nabla _{1}^{j}K^{\alpha }\left( x,y\right) \right\vert \leq
C_{CZ}\left\vert x-y\right\vert ^{\alpha -j-n-1},\ \ \ \ \ 0\leq j\leq
\kappa _{1}, \\
&&\left\vert \nabla _{1}^{\kappa }K^{\alpha }\left( x,y\right) -\nabla
_{1}^{\kappa }K^{\alpha }\left( x^{\prime },y\right) \right\vert \leq
C_{CZ}\left( \frac{\left\vert x-x^{\prime }\right\vert }{\left\vert
x-y\right\vert }\right) ^{\delta }\left\vert x-y\right\vert ^{\alpha -\kappa
_{1}-n-1},\ \ \ \ \ \frac{\left\vert x-x^{\prime }\right\vert }{\left\vert
x-y\right\vert }\leq \frac{1}{2},
\end{eqnarray*}%
and where the same inequalities hold for the adjoint kernel $K^{\alpha ,\ast
}\left( x,y\right) \equiv K^{\alpha }\left( y,x\right) $, in which $x$ and $%
y $ are interchanged, and where $\kappa _{1}$ is replaced by $\kappa _{2}$,
and $\nabla _{1}$\ by $\nabla _{2}$.

The $\kappa $\emph{-cube testing conditions} associated with an $\alpha $%
-fractional singular integral operator $T^{\alpha }$ introduced by Rahm,
Sawyer and Wick in \cite{RaSaWi} are given, with a slight modification, by%
\begin{eqnarray}
\left( \mathfrak{T}_{T^{\alpha }}^{\left( \kappa \right) }\left( \sigma
,\omega \right) \right) ^{2} &\equiv &\sup_{Q\in \mathcal{P}^{n}}\max_{0\leq
\left\vert \beta \right\vert <\kappa }\frac{1}{\left\vert Q\right\vert
_{\sigma }}\int_{Q}\left\vert T_{\sigma }^{\alpha }\left( \mathbf{1}%
_{Q}m_{Q}^{\beta }\right) \right\vert ^{2}\omega <\infty ,
\label{def Kappa polynomial} \\
\left( \mathfrak{T}_{\left( T^{\alpha }\right) ^{\ast }}^{\left( \kappa
\right) }\left( \omega ,\sigma \right) \right) ^{2} &\equiv &\sup_{Q\in 
\mathcal{P}^{n}}\max_{0\leq \left\vert \beta \right\vert <\kappa }\frac{1}{%
\left\vert Q\right\vert _{\omega }}\int_{Q}\left\vert \left( T_{\sigma
}^{\alpha }\right) ^{\ast }\left( \mathbf{1}_{Q}m_{Q}^{\beta }\right)
\right\vert ^{2}\sigma <\infty ,  \notag
\end{eqnarray}%
with $m_{Q}^{\beta }\left( x\right) \equiv \left( \frac{x-c_{Q}}{\ell \left(
Q\right) }\right) ^{\beta }$ for any cube $Q$ and multiindex $\beta $, where 
$c_{Q}$ is the center of the cube $Q$, and where as usual we interpret the
right hand sides as holding uniformly over all sufficiently smooth
truncations of $T^{\alpha }$. The following theorem provides a $Tp$
extension of the $T1$ theorem of David and Journ\'{e} \cite{DaJo} to a pair
of weights with one doubling and the other $A_{\infty }$ (and provided the
operator is bounded on unweighted $L^{2}\left( \mathbb{R}^{n}\right) $ when $%
\alpha =0$).

\begin{theorem}
\label{A infinity theorem}Suppose $0\leq \alpha <n$, and $\kappa _{1},\kappa
_{2}\in \mathbb{N}$ and $0<\delta <1$. Let $T^{\alpha }$ be an $\alpha $%
-fractional Calder\'{o}n-Zygmund singular integral operator on $\mathbb{R}%
^{n}$ with a standard $\left( \kappa _{1}+\delta ,\kappa _{2}+\delta \right) 
$-smooth $\alpha $-fractional kernel $K^{\alpha }$, and when $\alpha =0$,
suppose that $T^{0}$ is bounded on unweighted $L^{2}\left( \mathbb{R}%
^{n}\right) $. Assume that $\sigma $ and $\omega $ are locally finite
positive Borel doubling measures on $\mathbb{R}^{n}$ with doubling exponents 
$\theta _{1}$ and $\theta _{2}$ respectively. Set 
\begin{equation*}
T_{\sigma }^{\alpha }f=T^{\alpha }\left( f\sigma \right) 
\end{equation*}%
for any smooth truncation of $T^{\alpha }$. Suppose that $\kappa _{1}>\theta
_{1}+\alpha -n$ and $\kappa _{2}>\theta _{2}+\alpha -n$, and finally that in
addition, one of the measures is an $A_{\infty }$ weight. Then the operator $%
T_{\sigma }^{\alpha }$ is bounded from $L^{2}\left( \sigma \right) $ to $%
L^{2}\left( \omega \right) $, i.e. 
\begin{equation}
\left\Vert T_{\sigma }^{\alpha }f\right\Vert _{L^{2}\left( \omega \right)
}\leq \mathfrak{N}_{T^{\alpha }}\left( \sigma ,\omega \right) \left\Vert
f\right\Vert _{L^{2}\left( \sigma \right) },  \label{strong type}
\end{equation}%
uniformly in smooth truncations of $T^{\alpha }$, provided that the
classical fractional conditions (\ref{one-tailed}) of Muckenhoupt hold, and
the two dual $\kappa $-Cube Testing conditions (\ref{def Kappa polynomial})
hold. Moreover we have%
\begin{equation}
\mathfrak{N}_{T^{\alpha }}\left( \sigma ,\omega \right) \leq C\left( \sqrt{%
\mathcal{A}_{2}^{\alpha }\left( \sigma ,\omega \right) +\mathcal{A}%
_{2}^{\alpha }\left( \omega ,\sigma \right) }+\mathfrak{T}_{T^{\alpha
}}^{\left( \kappa \right) }\left( \sigma ,\omega \right) +\mathfrak{T}%
_{\left( T^{\alpha }\right) ^{\ast }}^{\left( \kappa \right) }\left( \omega
,\sigma \right) \right) ,  \label{more}
\end{equation}%
where the constant $C$ depends on $C_{CZ}$ in (\ref{sizeandsmoothness'}) and
the doubling parameters $\left( \beta _{1},\gamma _{1}\right) ,\left( \beta
_{2},\gamma _{2}\right) $ of the weights $\sigma $ and $\omega $, as well as
on the $A_{\infty }$ parameters of one of the weights. If $T^{\alpha }$ is
elliptic the inequality can be reversed.
\end{theorem}

\subsection{Optimal cancellation conditions}

Using the above theorem, we can now remove the compact set $E$ from the
characterization of optimal cancellation conditions in Theorem 5 in \cite%
{Saw2} provided that in addition one of the doubling measures is an $%
A_{\infty }$ weight (and $T^{0}$ is bounded on unweighted $L^{2}\left( 
\mathbb{R}^{n}\right) $ when $\alpha =0$). The proof of the next theorem is
a straightforward modification of the proof of Theorem 5 in \cite{Saw2} (in
turn a straightforward modification of that in Stein \cite[Theorem 4, page
306]{Ste2}), but now using Theorem \ref{A infinity theorem} above.

For $0\leq \alpha <n$, let $T^{\alpha }$ be a continuous linear map from
rapidly decreasing smooth test functions $\mathcal{S}$ to tempered
distributions in $\mathcal{S}^{\prime }$, to which is associated a kernel $%
K^{\alpha }\left( x,y\right) $, defined when $x\neq y$, that satisfies the
inequalities,%
\begin{equation}
\left\vert \partial _{x}^{\beta }\partial _{y}^{\gamma }K^{\alpha }\left(
x,y\right) \right\vert \leq A_{\alpha ,\beta ,\gamma ,n}\left\vert
x-y\right\vert ^{\alpha -n-\left\vert \beta \right\vert -\left\vert \gamma
\right\vert },\ \ \ \ \ \text{for all multiindices }\beta ,\gamma ;
\label{diff ineq}
\end{equation}%
such kernels are called \emph{smooth} $\alpha $-fractional Calder\'{o}%
n-Zygmund kernels on $\mathbb{R}^{n}$. An operator $T^{\alpha }$ is \emph{%
associated} with a kernel $K^{\alpha }$ if, whenever $f\in \mathcal{S}$ has
compact support, the tempered distribution $T^{\alpha }f$ can be identified,
in the complement of the support, with the function obtained by integration
with respect to the kernel, i.e.%
\begin{equation}
T^{\alpha }f\left( x\right) \equiv \int K^{\alpha }\left( x,y\right) f\left(
y\right) d\sigma \left( y\right) ,\ \ \ \ \ \text{for }x\in \mathbb{R}%
^{n}\setminus \func{Supp}f.  \label{identify}
\end{equation}

\begin{theorem}
Let $0\leq \alpha <n$ and $\kappa \in \mathbb{N}$. Suppose that $\sigma $
and $\omega $ are locally finite positive Borel doubling measures on $%
\mathbb{R}^{n}$ with doubling exponent $\theta $, where $\theta +\alpha
-n<\kappa $. Suppose also that the measure pair $\left( \sigma ,\omega
\right) $ satisfies the one-tailed conditions (\ref{one-tailed}) of
Muckenhoupt type, and that in addition, one of the measures is an $A_{\infty
}$ weight. Suppose finally that $K^{\alpha }\left( x,y\right) $ is a smooth $%
\alpha $-fractional Calder\'{o}n-Zygmund kernel on $\mathbb{R}^{n}$. In the
case $\alpha =0$, we also assume there is $T^{0}$ associated with the kernel 
$K^{0}$ that is bounded on unweighted $L^{2}\left( \mathbb{R}^{n}\right) $.%
\newline
Then there exists a bounded operator $T^{\alpha }:L^{2}\left( \sigma \right)
\rightarrow L^{2}\left( \omega \right) $, that is associated with the kernel 
$K^{\alpha }$ in the sense that (\ref{identify}) holds, \emph{if and only if}
there is a positive constant $\mathfrak{A}_{K^{\alpha }}^{\kappa }\left(
\sigma ,\omega \right) $ so that%
\begin{eqnarray}
&&\int_{\left\vert x-x_{0}\right\vert <N}\left\vert \int_{\varepsilon
<\left\vert x-y\right\vert <N}K^{\alpha }\left( x,y\right) \frac{p\left(
y\right) }{\left\Vert \mathbf{1}_{B\left( x_{0},N\right) }p\right\Vert
_{\infty }}d\sigma \left( y\right) \right\vert ^{2}d\omega \left( x\right)
\leq \mathfrak{A}_{K^{\alpha }}^{\kappa }\left( \sigma ,\omega \right) \
\int_{\left\vert x_{0}-y\right\vert <N}d\sigma \left( y\right) ,
\label{can cond} \\
&&\text{for all polynomials }p\text{ of degree less than }\kappa \text{, all 
}0<\varepsilon <N\text{ and }x_{0}\in \mathbb{R}^{n},  \notag
\end{eqnarray}%
along with a similar inequality with constant $\mathfrak{A}_{K^{\alpha ,\ast
}}^{\kappa }\left( \omega ,\sigma \right) $, in which the measures $\sigma $
and $\omega $ are interchanged and $K^{\alpha }\left( x,y\right) $ is
replaced by $K^{\alpha ,\ast }\left( x,y\right) =K^{\alpha }\left(
y,x\right) $. Moreover, if such $T^{\alpha }$ has minimal norm, then 
\begin{equation}
\left\Vert T^{\alpha }\right\Vert _{L^{2}\left( \sigma \right) \rightarrow
L^{2}\left( \omega \right) }\lesssim \mathfrak{A}_{K^{\alpha }}^{\kappa
}\left( \sigma ,\omega \right) +\mathfrak{A}_{K^{\alpha ,\ast }}^{\kappa
}\left( \omega ,\sigma \right) +\sqrt{\mathcal{A}_{2}^{\alpha }\left( \sigma
,\omega \right) +\mathcal{A}_{2}^{\alpha }\left( \omega ,\sigma \right) },
\label{can char}
\end{equation}%
with implied constant depending on $C_{CZ}$, $\theta $, $\kappa $ and the $%
A_{\infty }$ parameters of the $A_{\infty }$ weight. If $T^{\alpha }$ is
elliptic the inequality can be reversed.
\end{theorem}

\section{Concluding remarks}

The problem investigated in this paper and \cite{Saw2} is that of \emph{%
fixing a measure pair} $\left( \sigma ,\omega \right) $, and then asking for
a characterization of the $\alpha $-fractional $CZO$'s $T^{\alpha }$ that
are bounded from $L^{2}\left( \sigma \right) $ to $L^{2}\left( \omega
\right) $ - the first solution being the one weight case of Lebesgue measure
with $\alpha =0$ in \cite{DaJo}. This problem is in a sense `orthogonal' to
other recent investigations of two weight norm inequalities, in which one 
\emph{fixes the elliptic operator }$T^{\alpha }$, and asks for a
characterization of the weight pairs $\left( \sigma ,\omega \right) $ for
which $T^{\alpha }$ is bounded.

This latter investigation for a fixed operator is extraordinarily difficult,
with essentially just \textbf{one} CZ operator $T^{\alpha }$ known to have a
characterization of the weight pairs $\left( \sigma ,\omega \right) $,
namely the Hilbert transform on the line, see the two part paper \cite%
{LaSaShUr3};\cite{Lac}, and also \cite{SaShUr10} for an extension to
gradient elliptic operators on the line. In particular, matters appear to be
very bleak in higher dimensions due to the example in \cite{Saw1} which
shows that the energy side condition, used in virtually all attempted
characterizations, fails to be necessary for even the most basic elliptic
operators - the stronger pivotal condition is however shown in \cite{LaLi}
to be necessary for boundedness of the $g$-function, a \emph{Hilbert space }%
valued $CZO$ with a strong gradient positivity property, and the weight
pairs were then characterized in \cite{LaLi} by a single testing condition%
\footnote{%
The testing condition (1.3) in \cite{LaLi} implies the weights share no
common point masses, and then an argument in \cite{LaSaUr1} using the
asymmetric $A_{2}$ condition of Stein shows that the $A_{2}$ condition is
implied by the testing condition. Thus (1.3) can be dropped from the
statement of Theorem 1.2.}.

On the other hand, the problem for a fixed measure pair has proved more
tractable in that there are now many weight pairs $\left( \sigma ,\omega
\right) $ for which the characterization of boundedness of operators is
known. However, the techniques required for these results are taken largely
from investigations of the problem where the operator is fixed. In
particular, an adaptation of the `pivotal' argument in \cite{NTV4} to the
weighted Alpert wavelets in \cite{RaSaWi}, and a Parallel Corona
decomposition from \cite{LaSaShUr4} are used.

The questions of relaxing the $\kappa $-Cube Testing conditions, and the
side conditions of doubling and $A_{\infty }$, both remain open. There is in
fact no known example of a $CZO$ for which the $T1$ theorem fails.

In the case $\alpha =0$, there is the problem analogous to the celebrated `$%
A_{2}$ conjecture' solved in general in \cite{Hyt}, of determining the
optimal dependence of the above estimates on the $A_{2}$ characteristic. In
particular the dependence for the restricted weak type inequality should
follow using the pigeonholing and corona construction introduced in \cite%
{LaPeRe} and used in \cite{Hyt}.

\end{document}